\documentclass[10pt,a4paper]{amsart}
\usepackage{amsmath}
\usepackage{amsfonts}
\usepackage{amssymb}
\usepackage{amsthm}
\usepackage{graphicx}
\usepackage[left=3cm,right=3cm,top=3cm,bottom=3cm]{geometry}
\usepackage{verbatim}
\usepackage{pgf,tikz,pgfplots}
\usetikzlibrary{calc,angles,quotes}
\usetikzlibrary{hobby}
\usetikzlibrary{arrows}
\pgfplotsset{compat=1.15}
\usepackage{mathrsfs}
\usepackage{epsfig}
\usepackage{xspace}
\usepackage{hyperref,url}
\hypersetup{
    colorlinks=true,
    linkcolor=blue,
    filecolor=magenta,      
    urlcolor=cyan,
    citecolor=red
}
\usepackage{cite}
\usepackage{xstring}
\usepackage{intcalc}
\pgfplotsset{compat=1.15}
\usepackage{mathrsfs}
\usepackage{enumitem}
\usepackage{caption}

\usepackage[utf8]{inputenc}

\DeclareMathOperator{\interior}{int}

\DeclareMathOperator{\diam}{diam}

\DeclareMathOperator{\Aut}{Aut}

\newcommand{\cal}{\mathcal}
\newcommand{\R}{\mathbb{R}}

\newcommand{\C}{\mathbb{C}}
\newcommand{\Z}{\mathbb{Z}}

\newcommand{\PGL}{\mathrm{PGL}}
\newcommand{\PO}{\mathrm{PO}}

\DeclareMathOperator{\prox}{prox}

\DeclareMathOperator{\rel}{rel}

\newcommand{\PR}{\mathrm{P}}

\newcommand{\V}{{V}}

\newcommand{\ie}{i.e.\ }
\newcommand{\eg}{e.g.\ }
\newcommand{\resp}{resp.\ }

\newtheorem{prop}{Proposition}[section]
\newtheorem{thm}[prop]{Theorem}

\newtheorem{lemma}[prop]{Lemma}
\newtheorem{fait}[prop]{Fact}
\newtheorem{cor}[prop]{Corollary}

\theoremstyle{definition}
\newtheorem{defi}[prop]{Definition}

\theoremstyle{remark}

\numberwithin{equation}{section}

\author{Pierre-Louis Blayac}

\newcommand{\Addresses}{{% additional braces for segregating \footnotesize
  \bigskip
  
  \textsc{Laboratoire Alexander Grothendieck, Institut des Hautes \'Etudes Scientifiques, Universit\'e Paris-Saclay, 35 route de Chartres, 91440 Bures-sur-Yvette, France}\par
  \textit{E-mail address}: \texttt{pierre-louis.blayac@universite-paris-saclay.fr}
}}
\title{The boundary of rank-one divisible convex sets}

\begin{document}

\begin{abstract}
 We prove that for any non-symmetric irreducible divisible convex set, the proximal limit set is the full projective boundary.
\end{abstract}

\maketitle

\section{Introduction}

This note concerns the rich topic of divisible convex sets, which started more than sixty years ago with the work of Kuiper \cite{kuiper} and Benz\'ecri \cite{benz_varlocproj}, and is today very active. We refer to Benoist's survey \cite{benoist_survey}, which presents many interesting results and shows how diverse the mathematics interacting with this topic are. Let us fix for the whole paper a finite-dimensional real vector space $\V$. A subset of the projective space $\PR(\V)$ is \emph{properly convex} if it is convex and bounded in some affine chart. A properly convex open subset $\Omega\subset\PR(\V)$ is \emph{divisible} if it is \emph{divided} by some discrete subgroup of $\Gamma\subset\PGL(\V)$, \ie $\Gamma$ acts cocompactly on $\Omega$. We denote by $\Aut(\Omega)\subset\PGL(\V)$ the closed subgroup consisting of the elements $g$ that preserve $\Omega$.

\subsection{Structural results on divisible convex sets}\label{sec:structural result}

The result that we discuss here continues a line of structural results on divisible convex sets $\Omega$. These make the link between several kinds of regularity properties of the projective boundary $\partial\Omega\subset\PR(\V)$, of algebraic properties of $\Aut(\Omega)$ and its discrete cocompact subgroups, and of dynamical properties of the action of $\Aut(\Omega)$ and its subgroups on $\PR(\V)$.

One cornerstone of these structural results is the following result due to Vey \cite[Th.\,3]{vey}. Consider a divisible convex set $\Omega\subset\PR(\V)$. Then 
\begin{itemize}
 \item either there exists two proper subspaces $V_1,V_2\subset\V$ with $\V=V_1\oplus V_2$ and two properly convex open cones $C_1\subset V_1$ and $C_2\subset V_2$ such that $\PR(C_1)\subset\PR(V_1)$ and $\PR(C_2)\subset\PR(V_2)$ are divisible convex sets and $\Omega=\PR(C_1+C_2)$ --- in this case $\Omega$ is said to be \emph{reducible};
 \item or any cocompact closed subgroup of $\Aut(\Omega)$ is \emph{strongly irreducible}, in the sense that it does not preserve any finite union of proper subspaces of $\PR(\V)$ --- in this case $\Omega$ is said to be \emph{irreducible}.
\end{itemize}

Let us assume that $\Omega$ is irreducible. Combining work of Koecher \cite{koecher1999minnesota}, Vinberg \cite{Vinberg65} and Benoist \cite{CD2} yields the following dichotomy:
\begin{itemize}
 \item either $\Aut(\Omega)\subset\PGL(\V)$ is a semi-simple Lie subgroup that acts transitively on $\Omega$, in which case $\Omega$ is called \emph{symmetric};
 \item or $\Aut(\Omega)\subset\PGL(\V)$ is a discrete Zariski-dense subgroup.
\end{itemize}

If $\Omega$ is symmetric, then it naturally identifies with the Riemannian symmetric space of $\Aut(\Omega)$, and there is yet another natural dichotomy: namely, either  $\Aut(\Omega)$ has real rank $1$, in which case $\Omega$ is an ellipsoid and $\Aut(\Omega)$ is isomorphic to $\PO(n,1)$ for $n=\dim(\V)-1$, or $\Aut(\Omega)$ has real rank greater than one, it is isomorphic to $\PGL(n,\mathbb K)$ for some $n\geq 3$, and for $\mathbb K=\R$, $\C$, or the classical division algebra of quaternions, or of octonions if $n=3$ (see for instance \cite[\S2.4]{benoist_survey}).

Recently, A.\,Zimmer proved the following higher-rank rigidity result \cite[Th.\,1.4]{zimmer_higher_rank}, analogous to a celebrated result in Riemannian geometry by Ballmann \cite{ballmann_higher_rank} and Burns--Spatzier \cite{burns_spatzier_higher_rank}. If $\Omega$ is not symmetric, then it is \emph{rank-one} in the following sense.

\begin{defi}\label{def:rk1}
 A divisible convex set $\Omega\subset\PR(\V)$ is said to be rank-one if there exists in $\partial\Omega$ a \emph{strongly extremal point}, namely a point $\xi\in\partial\Omega$ such that $[\xi,\eta]\cap\Omega$ is non-empty for any $\eta\in\partial\Omega\smallsetminus\{\xi\}$ (in other words, $\xi$ is ``visible'' from any other point of the projective boundary).
\end{defi}

The notion of rank-one divisible convex sets (and more generally of rank-one geodesics, automorphisms, groups of automorphisms, quotients of properly convex open sets, which we do not define here)  was developped by M.\,Islam \cite{islam_rank_one} and Zimmer \cite{zimmer_higher_rank}, who established other characterisations of this property; see also \cite{topmixing,mesureBM} for more characterisations. 

It is elementary to check that reducible divisible convex sets and symmetric irreducible divisible convex sets with higher-rank automorphism groups are not rank-one (see \eg \cite[\S2.7 \& \S7]{topmixing}). These convex sets are hence called \emph{higher-rank}. On the other hand, ellipsoids are rank-one.

\subsection{The proximal limit set}\label{sec:prox}

Let $\Omega\subset\PR(\V)$ be an irreducible divisible convex set. The present note concerns an important $\Aut(\Omega)$-invariant compact subset of the projective boundary $\partial\Omega$, called the \emph{proximal limit set} and denoted by $\Lambda^{\prox}_\Omega$. Recall that a projective transformation $g\in\PGL(\V)$ is called \emph{proximal} if it has an attracting fixed point in $\PR(\V)$. 

\begin{defi}
 Let $\Omega\subset\PR(\V)$ be an irreducible divisible convex set. The proximal limit set of $\Omega$ is the closure of the set of attracting fixed points of proximal elements of $\Aut(\Omega)$.
\end{defi}

By work of Vey \cite[Prop.\,3]{vey} and Benoist \cite[Lem.\,3.6.ii]{BenoistPropAsymp}, the proximal limit set is also
\begin{itemize}
 \item the closure of the set of extremal points of $\overline\Omega$;
 \item the closure of the set of attracting fixed points of proximal elements of $\Gamma$, for any cocompact closed subgroup $\Gamma\subset\Aut(\Omega)$;
 \item the smallest (for inclusion) closed $\Gamma$-invariant non-empty  subset of $\PR(\V)$ for any cocompact closed subgroup $\Gamma\subset\Aut(\Omega)$.
\end{itemize}
If $\Omega$ is an ellipsoid, \ie a rank-one symmetric divisible convex set, then $\Lambda^{\prox}_\Omega=\partial\Omega$ and $\Aut(\Omega)$ acts transitively on it. If $\Omega$ is a higher-rank symmetric irreducible divisible convex set, then $\Lambda^{\prox}_\Gamma$ is an analytic submanifold of $\PR(\V)$ of dimension less than $\dim(\V)-2$, and hence is a proper subset of $\partial\Omega$ (see \cite[\S7]{topmixing}), on which $\Aut(\Omega)$ acts transitively.

Our goal is to prove the following result.

\begin{thm}\label{thm:limset=bord}
 Let $\Omega\subset\PR(\V)$ be a rank-one divisible convex set. Then $\Lambda^{\prox}_\Omega=\partial\Omega$.
\end{thm}

Combined with Zimmer's higher-rank rigidity theorem \cite[Th.\,1.4]{zimmer_higher_rank}, Theorem~\ref{thm:limset=bord} yields the following answer to a question of Benoist \cite[Prob.\,5]{exos_benoist}.

\begin{cor}
 Let $\Omega\subset\PR(\V)$ be a non-symmetric irreducible divisible convex set.  Then \linebreak $\Lambda^{\prox}_\Omega=\partial\Omega$.
\end{cor}

Let $\Omega$ be a rank-one divisible convex set. The conclusion of Theorem~\ref{thm:limset=bord} holds trivially if $\Omega$ is symmetric (\ie is an ellipsoid). Thus we may assume that $\Omega$ is not symmetric, hence that $\Aut(\Omega)$ is discrete and Zariski-dense in $\PGL(\V)$ (and finitely generated). 

Benoist \cite[Th.\,1.1]{CD1} proved that $\Aut(\Omega)$ is Gromov-hyperbolic if and only if $\Omega$ is \emph{strictly convex} (\ie all points of $\partial\Omega$ are extremal), if and only if $\partial\Omega$ is $\cal C^1$. In this case, strict convexity implies that $\Lambda_\Omega^{\prox}=\partial\Omega$. One may find in \cite{CD1} more precise results on the regularity of $\partial\Omega$.

Benoist \cite{CD4} also studied non-strictly convex 3-dimensional rank-one divisible convex sets. He constructed examples, and established a precise description of these which implies that $\Lambda^{\prox}_\Omega=\partial\Omega$.

Islam--Zimmer \cite{IZ19relhypb} generalised Benoist's description to higher-dimensional rank-one divisible convex sets, under the assumption that $\Aut(\Omega)$ is relatively hyperbolic, and their result implies that $\Lambda_\Omega^{\prox}=\partial\Omega$ in this case. M.\,Bobb \cite{Bobb} also generalised Benoist's result under the assumption that each non-trivial face of $\Omega$ (see Section~\ref{sec:faces}) is contained in a properly embedded simplex of dimension $\dim(\V)-2$, namely a closed simplex $S\subset \overline\Omega$ whose relative interior (see Section~\ref{sec:faces}) is exactly $S\cap\Omega$; Bobb's result also implies that $\Lambda_\Omega^{\prox}=\partial\Omega$.

\subsection{Organisation of the paper}

In Section~\ref{sec:reminder} we recall basic notions of projective geometry. In particular, we recall the definition of the Hilbert metric on $\Omega$, and how it naturally extends to the projective closure $\overline\Omega$.

In Section~\ref{sec:shadow} we establish a weak, convex projective version (Lemma~\ref{lem:ombre faible}) of Sullivan's celebrated  Shadow lemma. This result can be seen as a consequence of a more standard convex projective version of the Sullivan Shadow lemma proved in \cite[Lem.\,4.2]{mesureBM}, where we develop the theory of Patterson--Sullivan densities in convex projective geometry.

In Section~\ref{sec:grain} we establish two topological results (Lemmas~\ref{lem:zorn} and \ref{lemme_du_grain_de_sable}) which concern the arrangement of faces on the boundary of a convex set.

In Section~\ref{sec:demo} we use Sections~\ref{sec:shadow} and \ref{sec:grain} to prove Theorem~\ref{thm:limset=bord}.

{\bf Acknowledgements} I am grateful to Yves Benoist, Gilles Courtois, Olivier Glorieux and Fanny Kassel for helpful discussions. I thank Cyril Imbert and Jean-Baptiste Hiriart-Urruty for their help and references on convex analysis.

This project received funding from the European Research Council (ERC) under the European Union's Horizon 2020 research and innovation programme (ERC starting grant DiGGeS, grant agreement No 715982).

\section{Reminders in convex projective geometry}\label{sec:reminder}

\subsection{The Hilbert metric}
In the whole paper we fix a real vector space $\V=\R^{d+1}$, where $d\geq 1$. Let $\Omega\subset \PR(\V)$ be a properly convex open set. Recall that $\Omega$ admits an $\Aut(\Omega)$-invariant proper metric called the \emph{Hilbert metric} and defined by the following formula: for $(a,x,y,b)\in\partial\Omega\times\Omega\times\Omega\times\partial\Omega$ aligned in this order,
\begin{equation}\label{Equation : Hilbert}
d_\Omega(x,y)=\frac{1}{2}\log([a,x,y,b]),
\end{equation}
where $[a,x,y,b]$ is the cross-ratio of the four points, given by 
\begin{equation}\label{eq:birap}
 [a,x,y,b]=\frac{\Vert b-x\Vert\cdot \Vert a-y\Vert}{\Vert a-x\Vert\cdot \Vert b-y\Vert},
\end{equation}
where $\Vert\cdot\Vert$ is a norm on affine chart of $\PR(\V)$ containing $\overline\Omega$.

If $\Omega$ is an ellipsoid, then $(\Omega,d_\Omega)$ is the Klein model of the real hyperbolic space of dimension $d$; if $\Omega$ is a $d$-simplex, then $(\Omega,d_\Omega)$ is isometric to $\R^d$ endowed with a hexagonal norm.

\subsection{Faces of the boundary}\label{sec:faces}
Let us recall some basic notions about convexity. For any topological space $X$ and any subspace $Y$, we denote by $\interior_X(Y)$ (\resp $\partial_XY$) the interior (\resp boundary) of $Y$ with respect to $X$; if $X=\PR(\V)$, then we just write $\interior Y:=\interior_{\PR(V)}Y$ (\resp $\partial Y:=\partial_XY$) and call it the interior (\resp boundary) of $Y$.
Let $K\subset\PR(\V)$ be properly convex, \ie convex and bounded in some affine chart. 

\begin{itemize}
 \item The \emph{relative interior} (\resp \emph{relative boundary}) of $K$, denoted by $\interior_{\rel}(K)$ (\resp $\partial_{\rel}K$) is its topological interior (\resp boundary) with respect to the projective subspace it spans. 
 \item For $x\in \overline{K}$, the \emph{open face} of $x$ in $\overline{K}$, denoted by $F_{K}(x)$, consists of the points $y\in \overline{K}$ such that $[x,y]$ is contained in the relative interior of a (possibly trivial) segment contained in $\overline{K}$. The \emph{closed face} of $x$ is $\overline{F}_K(x)=\overline{F_K(x)}$. 
 \item A point $x\in\overline{K}$ is said to be \emph{extremal} (\resp \emph{strongly extremal}) if $F_K(x)=\{x\}$ (\resp $F_K(x)=\{x\}$  and $[x,y]\cap\interior_{\rel}K\neq\emptyset$ for $y\in\partial_{\rel} K\smallsetminus\{x\}$); one says that $K$ is \emph{strictly convex} if all the points in the relative boundary are extremal (and hence strongly extremal).
 \item Assume that $K$ spans $\PR(\V)$ and let $\xi\in\partial K$. A \emph{supporting hyperplane} of $K$ at $\xi$ is a hyperplane which contains $\xi$ but does not intersect $\interior (K)$. Note that there always exists such a hyperplane. 
\end{itemize}

\subsection{Extension of the Hilbert metric to the projective closure}
We extend the definition of the Hilbert distance $d_\Omega$ to pairs of points $x,y$ in the closure $\overline\Omega$. If $y$ is in the open face $F_\Omega(x)$ of $x$, then we set $d_{\overline\Omega}(x,y):=d_{F_\Omega(x)}(x,y)$, where $d_{F_\Omega(x)}$ is the Hilbert metric on $F_\Omega(x)$, seen as a properly convex open subset of the projective subspace it spans. If $y$ is not in $F_\Omega(x)$, then we set $d_{\overline\Omega}(x,y)=\infty$. 

For any $x\in\overline\Omega$ and $R>0$, we denote by $\overline B_{\overline\Omega}(x,R)$ (\resp $B_{\overline\Omega}(x,R)$) the set of points $y\in\overline\Omega$ with $d_{\overline\Omega}(x,y)\leq R$ (\resp $d_{\overline\Omega}(x,y)<R$). The following elementary fact plays an important role in this paper.

\begin{fait}\label{semi-continuite superieure}
 Let $\Omega\subset\PR(\V)$ be a properly convex open set. The function $d_{\overline\Omega}:\overline\Omega\times\overline\Omega\rightarrow\R\cup\{\infty\}$ is lower semi-continuous. As a consequence, for any $R>0$, the map
 \begin{equation*}
  \begin{array}{lllll}
   \overline{B}_{\overline{\Omega}}(\cdot,R) & : & \overline{\Omega} & \longrightarrow & \{\text{compact subsets of }\overline{\Omega}\}\\
   && \xi & \longmapsto & \overline{B}_{\overline{\Omega}}(\xi,R)
  \end{array}
 \end{equation*}
 is upper semi-continuous in the following sense: all accumulation points of $\overline{B}_{\overline{\Omega}}(\eta,R)$ when $\eta\to\xi$  for the Hausdorff topology must be contained in $\overline{B}_{\overline{\Omega}}(\xi,R)$.
\end{fait}
\begin{proof}
 Let $(x_n,y_n)_n$ converge to $(x,y)$ in $\overline\Omega{}^2$ and be such that $(d_{\overline\Omega}(x_n,y_n))_n$ converges; let us show that the limit is at least $d_{\overline\Omega}(x,y)$. We may assume that $x\neq y$ and $x_n\neq y_n$ for all $n$. For each $n$, let $a_n,b_n\in\partial\Omega$ (\resp $a,b\in\partial\Omega$) be such that $a_n,x_n,y_n,b_n$ (\resp $a,x,y,b$) are aligned in this order and $[a_n,b_n]$ (\resp $[a,b]$) is maximal for inclusion among segments of $\overline\Omega$; by definition $d_{\overline\Omega}(x_n,y_n)=\log[a_n,x_n,y_n,b_n]/2$ and $d_{\overline\Omega}(x,y)= \log[a,x,y,b]/2$, where we set $[a,x,y,b]=\infty$ if $a=x$ or $b=y$. Up to extracting, we may assume that $(a_n,b_n)_n$ converges to some $(a',b')\in\partial\Omega^2$. Since $[a,b]$ is maximal in $\overline\Omega$, it contains $[a',b']$, and $a,a',x,y,b',b$ are aligned in this order. The following concludes the proof:
 \[[a_n,x_n,y_n,b_n]\underset{n\to\infty}{\longrightarrow}[a',x,y,b']\geq [a,x,y,b].\qedhere\]
\end{proof}

We will also need the following elementary fact.

\begin{fait}\label{fait:euclVShilb}
 Let $\Omega\subset\PR(\V)$ be a properly convex open set and $\mathbb A\subset\PR(\V)$ an affine chart containing $\overline\Omega$ and equipped with some norm, with induced metric $d_{\mathbb A}$.
 For all $a\in\mathbb{A}$ and $t>0$, we denote by $\mathrm h_a^t$ the homothety of $\mathbb{A}$ with centre $a$ and ratio $t$. 
 Consider $x\in\overline\Omega$ and $0<r<R$. Then
 \begin{enumerate}
  \item \label{item:faceVSboule} $\overline{F}_\Omega(x) \subset \mathrm h_x^\lambda\left(\overline{B}_{\overline{\Omega}}(x,r)\right)$, where $$\lambda=\frac{\diam_{\mathbb A}(F_\Omega(x))(e^{2r}+1)}{d_{\mathbb A}(x,\partial_{\rel}F_\Omega(x))(e^{2r}-1)}>1;$$
  \item \label{item:bouleVSboule} $\mathrm h_x^{\mu}\left(\overline{B}_{\overline{\Omega}}(x,r)\right)\subset\overline{B}_{\overline{\Omega}}(x,R)$ where $\mu=(e^{2R}-1)/(e^{2r}-1)>1$.
 \end{enumerate} 
\end{fait}
\begin{proof}
 We see $\mathbb A$ as a vector space by setting $x=0$. Let $y\in\partial_{\rel} B_{\overline\Omega}(x,r)$, and consider $a>0$ and $b>1$ such that $-ay$ and $by$ lie in $\partial_{\rel}F_\Omega(x)$. To establish \eqref{item:faceVSboule}, it is enough to prove that 
 $$b\leq \frac{\max(a,b)(e^{2r}+1)}{\min(a,b)(e^{2r}-1)}.$$
 This is an immediate consequence of \eqref{eq:birap}, which implies that $(a+1)b=e^{2r}a(b-1)$, hence that 
 $$b=\frac{ae^{2r}+b}{a(e^{2r}-1)}.$$
 Consider $t\in (1,b)$ such that $ty\in\partial_{\rel} B_{\overline\Omega}(x,R)$. By \eqref{eq:birap}, we have
 \begin{equation*}
  1=\frac{ab(e^{2r}-1)}{ae^{2r}+b} \qquad \text{and} \qquad t=\frac{ab(e^{2R}-1)}{ae^{2R}+b}.
 \end{equation*}
 Thus, 
 \begin{equation*}
  t=\frac{(e^{2R}-1)(ae^{2r}+b)}{(e^{2r}-1)(ae^{2R}+b)}>\frac{e^{2R}-1}{e^{2r}-1},
 \end{equation*}
 and this proves \eqref{item:bouleVSboule}.
\end{proof}

\section{A weak Shadow lemma}\label{sec:shadow}

Let $\Omega\subset\PR(V)$ be a properly convex open set. For $x\in\overline\Omega$, $y\in\Omega$ and $R>0$, we consider the set
\begin{equation*}
\begin{matrix}
 \mathcal{O}_R(x,y)  = \{\xi\in\partial\Omega:[x,\xi]\cap B_\Omega(y,R)\neq\emptyset\},
\end{matrix}
\end{equation*}
which we interpret as the shadow cast on $\partial\Omega$ by the balls of radius $R$ around $y$ with a light source at $x$.

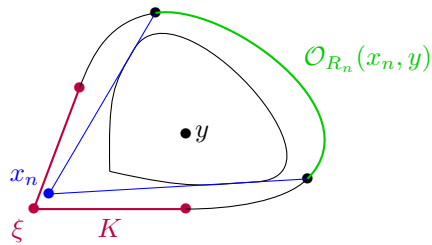
\begin{figure}
 \centering
 \begin{tikzpicture}[scale=2]
  %les coordonnées des points du bord
  \coordinate (xi) at (0,0);
  \coordinate (a) at (1,0);
  \coordinate (b) at (1.8,.2);
  \coordinate (c) at (.8,1.3);
  \coordinate (d) at (.3,.8);
  %les coordonnées des points à l'intérieur
  \coordinate (x) at (.1,.1);
  \coordinate (o) at (1,.5);
  \coordinate (y) at ($(x)!.8!(b)$);
  \coordinate (z) at ($(x)!.8!(c)$);
  \coordinate (xi') at ($(xi)!.5!(o)$);
  %les vecteurs tangents
  %au bord de Omega
  \coordinate (Ta) at ($.3*(a)-.3*(xi)$);
  \coordinate (Tb) at (.3,.3);
  \coordinate (Tc) at (-.3,-.05);
  \coordinate (Td) at ($.3*(xi)-.3*(d)$);
  %sur la sphère centré en o
  \coordinate (Ty) at ($(b)-(y)$);
  \coordinate (Tz) at ($(z)-(c)$);
  \coordinate (Txi') at ($.5*(a)-.5*(xi')+(Ta)$);
  \coordinate (Tmxi') at ($.5*(xi')-.5*(d)+(Td)$);
   
  %dessiner les points
  \draw (xi) node[purple]{$\bullet$} node[below left,purple]{$\xi$};
  \draw (a) node[purple]{$\bullet$};
  \draw (b) node{$\bullet$};
  \draw (c) node{$\bullet$};
  \draw (d) node[purple]{$\bullet$};
  \draw (o) node{$\bullet$} node[right]{$y$};
  \draw (x) node[blue]{$\bullet$} node[above left,blue]{$x_n$};
%  \draw (y) node{$\bullet$};
%  \draw (z) node{$\bullet$};
   
  %dessiner le bord de Omega
  \draw [thick, purple] (xi)--(a) node[below, midway, purple]{$K$};
  \draw (a) .. controls ($(a)+(Ta)$) and ($(b)-.5*(Tb)$) .. (b);
  \draw [thick, green!80!black] (b).. controls ($(b)+1.5*(Tb)$) and ($(c)-1.5*(Tc)$) ..(c);
  \draw (c).. controls ($(c)+(Tc)$) and ($(d)-(Td)$) ..(d);
  \draw [thick, purple] (d)--(xi);
  %dessiner la sphère centrée en o
  \draw (y)..controls ($(y)+2*(Ty)$) and ($(z)-2*(Tz)$)..(z);
  \draw (z)..controls ($(z)+(Tz)$) and ($(xi')-(Tmxi')$)..(xi');
  \draw (xi')..controls ($(xi')+(Txi')$) and ($(y)-(Ty)$)..(y);
  %dessiner tangence de la sphère
  \draw [blue] (x)--(b);
  \draw [blue] (x)--(c);
   
  %écrire l'ombre
  \draw (2.2,1) node[green!80!black]{$\mathcal{O}_{R_n}(x_n,y)$};
 \end{tikzpicture}
 \caption{Illustration of the proof of Lemma~\ref{lem:ombre faible}. The green shadow $\mathcal{O}_{R_n}(x_n,y)$ fills more and more of $\partial\Omega\smallsetminus K$ as $n$ tends to infinity}
 \label{fig:lemombrebis}
\end{figure}

\begin{lemma}\label{lem:ombre faible}
 Let $\Omega\subset\PR(V)$ be a rank-one divisible convex set. Then there exists $R>0$ such that $\mathcal O_R(x,y)$ contains a point of the proximal limit set $\Lambda_\Omega^{\prox}$ (see Section~\ref{sec:prox}) for all $x,y\in\Omega$. 
\end{lemma}
\begin{proof}
Recall from Section~\ref{sec:prox} that $\Lambda^{\prox}_\Omega$ is the closure of the set of extremal points of $\partial\Omega$.
By contradiction, suppose that there is a diverging sequence of positive numbers $(R_n)_n$ and sequences of points $(x_n)_n,(y_n)_n$ in $\Omega$ such that for any $n\geq 0$, the set $\mathcal O_{R_n}(x_n,y_n)$ does not contain any extremal point of $\partial\Omega$. Since $\Omega$ is divisible, $\Aut(\Omega)$ acts cocompactly on $\Omega$, and so we may assume that $(y_n)_n$ remains in a compact subset of $\Omega$, and up to extracting, we may further assume that $(y_n)_n$ converges to a point $y\in\Omega$. Up to replacing $R_n$ by $R_n-d_\Omega(y_n,y)$, we may actually assume that $(y_n)_n$ is constant equal to $y$.

Up to extraction, we assume that $(x_n)_n$ converges to some $\xi\in\overline{\Omega}$. If $\xi\in\Omega$, then for $n$ such that ${R_n}\geq d_\Omega(o,\xi)+1$ and $d_\Omega(x_n,\xi)<1$, we have $\mathcal{O}_{R_n}(x_n,y)=\partial\Omega$, which is absurd; hence $\xi\in\partial\Omega$. 

Let $K\subset\partial\Omega$ be the set of points $\eta$ such that $[\xi,\eta]\subset\partial\Omega$. Then
\[\partial\Omega\smallsetminus K\subset \bigcup_n\bigcap_{k\geq n}\mathcal{O}_{R_k}(x_k,y).\]
See Figure~\ref{fig:lemombrebis}. Let $\eta\in\partial\Omega\smallsetminus K$, and $z\in[\xi,\eta]\cap\Omega$. Since $(x_n)_n$ converges to $\xi$, we can find $z_n\in[x_n,\eta]\cap B_\Omega(z,1)$ for any large enough $n$. On the other hand, $R_n\geq d_\Omega(y,z)+2$ for $n$ large. Thus, $z_n\in B_\Omega(y,R_n)$ and hence $\eta\in\mathcal O_{R_n}(x_n,y)$ for any large enough $n$.
 
By assumption, this implies that all extremal points are contained in $K$. Since $\Omega$ is rank-one (see Definition~\ref{def:rk1}) and $\Aut(\Omega)$ is irreducible, $\partial\Omega$ contains a strongly extremal point which is different from $\xi$. Such a point cannot lie in $K$; this yields a contradiction.
\end{proof}

\section{Two lemmas on general properly convex open sets}\label{sec:grain}

In this section we prove two lemmas on the arrangement of faces on the boundary of a general properly convex open subset of $\PR(\V)$, which is not necessarily divisible.

Let us first give a family of examples of non-divisible properly convex open subsets of $\PR(\R^4)$ that one may wish to keep in mind while reading the lemmas of this section.
Let $f:\R\rightarrow [1,\infty)$ be a $2\pi$-periodic upper semi-continuous function. Let $\Omega_f$ be the interior of the convex hull in $\R^3$ of 
\begin{equation}\label{eq:cylbiz}
\{(\cos(\theta),\sin(\theta),f(\theta)) : \theta\in\R\}\cup \{(\cos(\theta),\sin(\theta),-f(\theta)) : \theta\in\R\}.
\end{equation}
Since $f$ is upper semi-continuous and $2\pi$-periodic, it is bounded, and so is $\Omega_f$.
Let us identify $\R^3$ with an affine chart of $\PR(\R^4)$, so that $\Omega_f$ is a properly convex open subset of $\PR(\R^4)$. One can check that \eqref{eq:cylbiz} is exactly the set of extremal points of $\Omega_f$, and that for any $\theta\in \R$, the set $\{(\cos(\theta),\sin(\theta),z) : z\in(-f(\theta),f(\theta))\}$ is an open face of of $\Omega_f$.

\begin{figure}
\centering
\includegraphics[scale=.5]{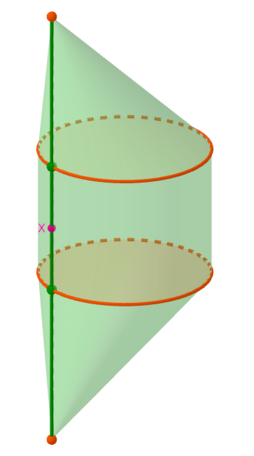}
\includegraphics[scale=.25]{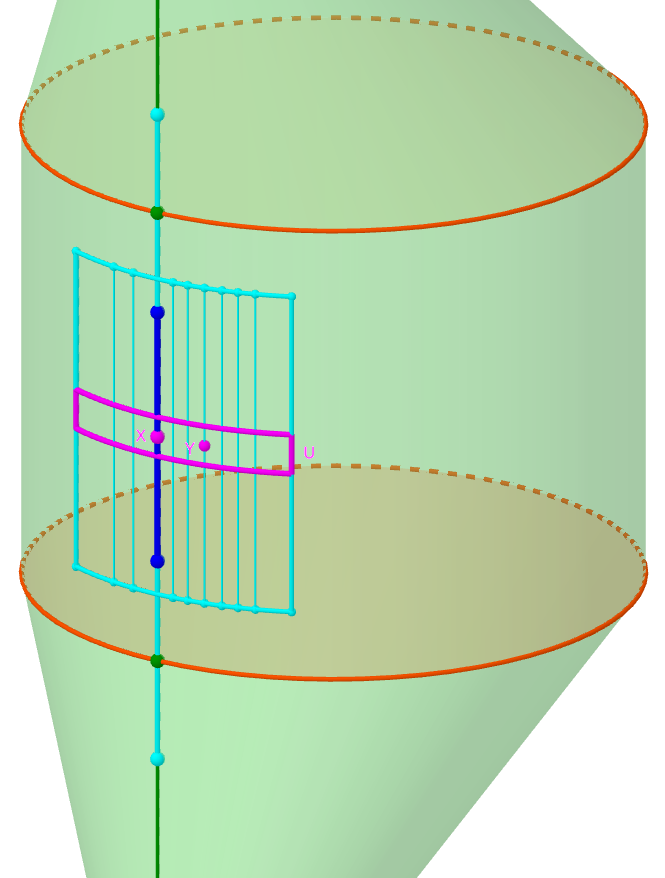}
\caption{On the left: the whole set $\Omega_f\subset\PR(\R^4)$ as in Section~\ref{sec:grain} for a $2\pi$-periodic function $f$ which is constant on $\R\smallsetminus 2\pi\Z$ and discontinuous on $2\pi\Z$. On the right: a zoom on $\Omega_f$ at point $x$, with blue vertical segments representing several $d_{\overline\Omega_f}$-balls.}
\label{fig:omegaf}
\end{figure}

% \subsection{Proof of Lemma~\ref{lemme_du_grain_de_sable}}
% Before we dive into the proof, let us first make the following elementary observation. If $\xi\in U$ is an extremal point of $\partial\Omega$, then $\overline B_{\overline\Omega}(\xi,R_0)=\{\xi\}$, and hence it lies in the interior of any compact neighbourhood $A$ of $\xi$, which in turn is a subset of $\interior_{\partial\Omega}(B_{\overline{\Omega}}(A,R))$. Unfortunately, $U$ does not always contain an extremal point (see for instance cases where $\Omega$ is a polytope, or a cylinder, or is the projective model of a higher-rank symmetric space). We are going to see that the conclusion of Lemma~\ref{lemme_du_grain_de_sable} holds if we consider a point $\xi\in U$ that has a small face, in the sense that $\overline B_{\overline\Omega}(\xi,\frac{R_0+R}2)$ is contained in any accumulation point for the Hausdorff topology of $\overline B_{\overline\Omega}(\eta,R)$ when $\eta$ tends to $\xi$; this corresponds to Step 6 below, where the formulation is different because we identify $U$ with an open subset of an affine space (Step 1). It then remains to check that $U$ always contains such a point: this is the goal of Steps 2-5.

\subsection{Existence of a point on the boundary with a sufficiently small Hilbert ball}

Let $\Omega\subset\PR(\V)$ be a properly convex open set.
We saw in Fact~\ref{semi-continuite superieure} that, for any $R>0$, the map $\overline B_{\overline\Omega}(\cdot,R)$ is upper semi-continuous on $\overline\Omega$. However, it is not continuous in general. For instance, in Figure~\ref{fig:omegaf} on the left, each orange point $x\in\partial\Omega_f$ is extremal, hence $\overline B_{\overline \Omega_f}(x,R)=\{x\}$ for any $R>0$, and orange points accumulate to a green point $y$ which has a non-trivial face, hence $\overline B_{\overline\Omega_f}(y,R)\neq\{y\}$, and so $\overline B_{\overline\Omega_f}(\cdot,R)$ is discontinuous at $y$.

The goal of the next lemma is to show that in any open subset of $\partial\Omega$, one can find a point at which $\overline B_{\overline\Omega}(\cdot,R)$ is ``almost continuous''.

%For a general $\Omega$, we may set $\liminf_{y\to x}\overline B_{\overline\Omega}(y,R)$ (\resp $\limsup_{y\to x}\overline B_{\overline\Omega}(y,R)$) to be the intersection (\resp union) of all accumulation points of $\overline B_{\overline\Omega}(\cdot,R)$ at $x$.

\begin{lemma}\label{lem:zorn}
 Let $\Omega\subset\PR(\V)$ be a properly convex open set, $0<r<R$ and $U\subset\partial\Omega$ a non-empty open subset. Then one can find a point $x\in U$ such that $\overline{B}_{\overline{\Omega}}(x,r)$ is contained in any accumulation point of $\overline{B}_{\overline{\Omega}}(y,R)$ (for the Haudorff topology) when $y$ tends to $x$.%, \ie
% $$\overline{B}_{\overline{\Omega}}(x,r)\subset \liminf_{y\to x}\overline{B}_{\overline{\Omega}}(y,R) \subset \limsup_{y\to x}\overline{B}_{\overline{\Omega}}(y,R) \subset \overline{B}_{\overline{\Omega}}(x,R).$$
\end{lemma}

Note that if $x\in\partial\Omega$ is an extremal point, then $\overline B_{\overline\Omega}(x,r)=\{x\}$, and so $\overline B_{\overline\Omega}(\cdot,R)$ is continuous at $x$. Thus, the lemma is immediate when $U$ contains an extremal point.

Suppose $\Omega=\Omega_f$ for some $2\pi$-periodic upper semi-continuous function $f:\R\rightarrow [1,\infty)$, consider the open subset $U=\{(\cos(\theta),\sin(\theta),z) : z\in (-1,1), \ \theta\in\R\}\subset\partial\Omega$, and consider $\theta\in (0,2\pi)$, $z\in(-1,1)$ and $x=(\cos(\theta),\sin(\theta),z)$. Fix $R>0$. Then $\overline B_{\overline\Omega}(\cdot,R)$ is continuous at $x$ if and only if $f$ is continuous at $\theta$. In particular, if $f$ is discontinuous everywhere, then $\overline B_{\overline\Omega_f}(\cdot,R)$ is discontinuous everywhere on $U$.
Proving Lemma~\ref{lem:zorn} in the case $\Omega=\Omega_f$ roughly amounts to proving that  for any $\epsilon>0$, we can find $\theta_\epsilon\in\R$ at which $f$ is ``$\epsilon$-almost continuous'', \ie such that $$f(\theta_\epsilon)-\epsilon\leq \liminf_{\theta\to\theta_\epsilon}f(\theta)\leq \limsup_{\theta\to\theta_\epsilon}f(\theta)\leq f(\theta_\epsilon).$$

\begin{proof}
Fix an affine chart $\mathbb A$ that contains $\overline\Omega$, and a norm on $\mathbb A$ whose associated metric is denoted by $d_{\mathbb A}$, with associated balls denoted by $B_{\mathbb A}(x,t)$ for $x\in \mathbb A$ and $t>0$. For the rest of this proof, we set $B_t(x)=\overline B_{\overline\Omega}(x,t)$ for $x\in\overline\Omega$ and $t>0$,  and denote by $\cal B_t(x)$ the set of accumulation points (for the Hausdorff topology) of $B_t(y)$ when $y$ tends to $x$.

{\bf First step:} We reduce $U$ to control the dimension of faces.
 
 Let $k$ be the largest integer such that $\{ x\in U : \dim F_\Omega(x)\geq k\}$ has non-empty interior in $U$. Let $x_0\in U$ and $\epsilon>0$ be such that $\overline{B}_{\mathbb{A}}(x_0,2\epsilon)\cap \partial\Omega$ is contained in this interior, and $\dim F_\Omega(x_0)=k$. Note that $D:=\{x:\dim F_\Omega(x)=k\}\cap \overline{B}_{\mathbb{A}}(x_0,\epsilon)\cap \partial\Omega$ is dense in $U':=\overline{B}_{\mathbb{A}}(x_0,\epsilon)\cap \partial\Omega$. Up to taking $\epsilon$ even smaller, we can assume that $\diam_{\mathbb A}\overline\Omega\leq \epsilon^{-1}$.
 
{\bf Second step:} We bound from below the size of faces of dimension $k$.
 
 Consider for this step $x\in D$. We denote by $\mathbb A_x$ the affine subspace of $\mathbb A$ spanned by $F_\Omega(x)$, which has dimension $k$. Any point in $\partial_{\rel}F_\Omega(x)$ has a face of dimension strictly less that $k$, hence is not in $\overline{B}_{\mathbb{A}}(x_0,2\epsilon)$ by definition of $x_0$ and $\epsilon$. 
 By triangular inequality, this implies that $$B_{\mathbb A}(x,\epsilon)\cap\mathbb A_x \subset F_\Omega(x)\subset \mathbb A_x.$$
 
 Set $\lambda:=\epsilon^{-2}(e^{2R}+1)/(e^{2R}-1)>1$. 
 For all $a\in\mathbb{A}$ and $t>0$, we denote by $\mathrm h_a^t$ the homothety of $\mathbb{A}$ with centre $a$ and ratio $t$. By Fact~\ref{fait:euclVShilb}.\ref{item:faceVSboule}, 
 \[\overline{F}_\Omega(x) \subset \mathrm h_x^\lambda\left(B_R(x)\right).\]
 As a consequence, we have 
 \begin{equation}\label{inegalite1}
 B_{\mathbb A}(x,\epsilon/\lambda)\cap\mathbb A_x \subset B_R(x)\subset \mathbb A_x.
 \end{equation}

 By upper semi-continuity of $B_R$ (Fact~\ref{semi-continuite superieure}) and the above \eqref{inegalite1}, any accumulation point of $B_{\mathbb A}(y,\epsilon/\lambda)\cap\mathbb A_y$ (for the Hausdorff topology) when $y\in D$ tends to $x$ is contained in $B_R(x)\subset \mathbb A_x$. One may easily deduce that the map $y\in D \mapsto \overline B_{\mathbb A}(y,\epsilon/\lambda)\cap\mathbb A_y$ is continuous for the Hausdorff topology.
 
 By upper semi-continuity of $B_R$ and density of $D$, any element $K\in\cal B_R(x)$ contains the limit of some sequence $(B_R(x_n))_n$ where $(x_n)_n\subset D$ converges to $x$. By \eqref{inegalite1}, this implies that
 \begin{equation}\label{inegalite1'}
 B_{\mathbb A}(x,\epsilon/\lambda)\cap\mathbb A_x \subset K\subset B_R(x)\subset  \mathbb A_x,
 \end{equation}
 hence $K$ has dimension $k$.
 
{\bf Third step:} We find a minimal element in $\cal B_R(x_0)$.
 
 Let us show that $\cal B_R(x_0)$ contains an element which is minimal for inclusion; by the Zorn lemma, it is enough to show that for every totally ordered subset $\cal A\subset \cal B_R(x_0)$, the intersection $K$ of all elements of $\cal A$ belongs to $\cal B_R(x_0)$.
 
 The Hausdorff topology on the set of compact subsets of $\PR(\V)$ is metrisable, and $K$ is in the closure of $\cal A$, so we can find a sequence $(K_n)_n$ in $\cal A$ that converges to $K$. If $K_n=K$ for some $n$, then $K\in\cal B_R(x_0)$; let us assume the contrary.
 For any $n$, we can find $m>n$ such that $K_m\subset K_n$ since, otherwise, $K\subsetneq K_n\subset K_m$ for any $m>n$ so $(K_m)_m$ would not converge to $K$.
 Thus, up to extraction, we may assume that $(K_n)_n$ is non-increasing.
 
 For each $n$, let $(x_{n,k})_k$ be a sequence converging to $x_0$ such that $(B_R(x_{n,k}))_k$ converges to $K_n$. Then $(B_R(x_{n,n}))_n$ converges to $K$, which thus belongs to $\cal B_R(x_0)$.
 
 Let $K\in \cal B_R(x_0)$ be a minimal element for inclusion, and let $(x_n)_n$ be a sequence in $U'$ converging to $x_0$ such that $(B_R(x_n))_n$ converges to $K$. By density of $D$ in $U'$, upper semi-continuity of $B_R$ and minimality of $K$, we may assume that $(x_n)_n$ is in $D$.

{\bf Fourth step:} We prove that $B_r(x_n)$ is contained in any element of $\cal B_R(x_n)$ for $n$ large enough.
 
 Assume by contradiction that for each $n$ there exists $K_n\in\cal B_R(x_n)$ that does not contain $B_r(x_n)$; since, by the previous step, $K_n$ and $B_r(x_n)$ are convex subsets of $\mathbb A_{x_n}$ that contain $x_n$ in their interior relative to $\mathbb A_{x_n}$, we may find $y_n\in\partial_{\rel}K_n\cap B_r(x_n)$.
 
 Up to extraction, we can assume that $(K_n)_n$ converges to some $K'$ and $(y_n)_n$ converges to some $y$. One can check that $K'\in \cal B_R(x)$. By \eqref{inegalite1'}, the compact convex sets $K'$ and $\{K_n\}_n$ have dimension $k$. According to the following classical and elementary fact, $y$ belongs to $\partial_{\rel}K'$.
 \begin{fait}
  If $(A_n)_n$ is a sequence of $k$-dimensional compact convex subsets of $\mathbb A$ that converges to a $k$-dimensional compact convex set $A$ for the Hausdorff topology, then $(\partial_{\rel}A_n)_n$ converges to $\partial_{\rel}A$.
 \end{fait}
 
 That $K_n\subset{B}_R(x_n)$ for each $n$ implies that $K'\subset K$, which in turn implies, by minimality of $K$, and because $K'\in\cal B_R(x)$, that $K'=K$. 
 
 Let $\mu=(e^{2R}-1)/(e^{2r}-1)>1$. By Fact~\ref{fait:euclVShilb}.\ref{item:bouleVSboule}, since $y_n\in B_r(x_n)$ for each $n$, we have $\mathrm h_{x_n}^\mu y_n\in B_R(x_n)$. As a consequence, $\mathrm h^\mu_{x_0}y\in K$, which contradicts the fact that $y\in\partial_{\rel} K$, $x_0\in\interior_{\rel}K$ and $\mu>1$.
\end{proof}

\subsection{The Grain of sand lemma}

Consider a properly convex open set $\Omega\subset\PR(\V)$, positive numbers $r<R$, a point $x\in\partial\Omega$ at which $\overline B_{\overline\Omega}(\cdot,R)$ is ``almost continuous'' in the sense of Lemma~\ref{lem:zorn}, and a compact neighbourhood $U$ of $x$ in $\partial\Omega$.

The Grain of sand lemma (Lemma~\ref{lemme_du_grain_de_sable}) says that the collection of balls $\overline{B}_{\overline{\Omega}}(y,R)$ centred at points $y\in U$ ``foliates'' a neighbourhood of ${B}_{\overline{\Omega}}(x,r)$, \ie that no ``grain of sand'' is inserted between the convex ``leaves'' of this ``foliation''. 

To illustrate this idea, we use again Figure~\ref{fig:omegaf}, which represents the set $\Omega=\Omega_f$ (defined at the beginning of Section~\ref{sec:grain}) for a $2\pi$-periodic function $f$ which is constant on $\R\smallsetminus 2\pi\Z$ and discontinuous on $2\pi\Z$. 
On the right of the figure, $U\subset\partial\Omega_f$ is the compact neighbourhood of the pink point $x$ which is delimited by the pink  rectangle on the cylinder. The vertical light blue segments are $d_{\overline\Omega_f}$-balls of radius $R$ centred at points of $U$, while the dark blue segment is a ball of radius $r\in (0,R)$ centred at $x$.
The union $B_{\overline\Omega_f}(U,R)$ of the balls $(B_{\overline\Omega_f}(y,R))_{y\in U}$ is  the region delimited by the light blue rectangle, to which one must add the tall central light blue vertical segment. The set $B_{\overline\Omega_f}(U,R)$ is not open in $\partial\Omega$. Its relative interior $\interior_{\partial\Omega}(B_{\overline\Omega_f}(U,R))$ is the region delimited by the light blue rectangle, and it is foliated by light blue balls for $d_{\overline\Omega_f}$. This relative interior contains the ball ${B}_{\overline{\Omega}_f}(x,r)$.

\begin{lemma}\label{lemme_du_grain_de_sable}
 Let $\Omega\subset\PR(\V)$ be a properly convex open set, $0<r<R$ and $x\in \partial\Omega$ such that $\overline{B}_{\overline{\Omega}}(x,r)$ is contained in any accumulation point of $\overline{B}_{\overline{\Omega}}(y,R)$ (for the Hausdorff topology) when $y$ tends to $x$. Then for any compact neighbourhood $U\subset \partial\Omega$ of $x$,
 \[ {B}_{\overline{\Omega}}(x,r)\subset \interior_{\partial\Omega}(B_{\overline{\Omega}}(U,R)),\]
 where $B_{\overline\Omega}(U,R):=\bigcup_{y\in U}B_{\overline\Omega}(y,R)$ is the uniform $R$-neighbourhood of $U$ for the metric $d_{\overline\Omega}$.
\end{lemma}

As in the previous section, the lemma holds trivially if $x$ is extremal, since 
$${B}_{\overline{\Omega}}(x,r)=\{x\}\subset \interior_{\partial\Omega}U\subset \interior_{\partial\Omega}(B_{\overline{\Omega}}(U,R)).$$

If $x$ is not extremal, then the situation is more delicate. In fact, the problem is related to the Invariance of Domain theorem. For instance, Lemma~\ref{lemme_du_grain_de_sable} is a consequence of this classical theorem under the assumption that there exists a neighbourhood $U'$ of $x$ such that $\dim(F_\Omega(y))=\dim(F_\Omega(x))$ for any $y\in U'$ (more details on how to apply the Invariance of Domain in this particular case are given in the following proof). This assumption is satisfied when $\Omega=\Omega_f$ for some $2\pi$-periodic upper semi-continuous function $f:\R\rightarrow [1,\infty)$, and $x=(\cos(\theta),\sin(\theta),z)$ for some $\theta\in\R$ and $z\in (-1,1)$.

In the general case, the strategy of proof of Lemma~\ref{lemme_du_grain_de_sable} is similar to one of those of the Invariance of Domain theorem.

\begin{proof}
 We first embed $U$ into a hyperplane of $\PR(\V)$.

 Let $\mathbb{A}$ be an affine chart of $\PR(\V)$ containing $\overline{\Omega}$. Let $\PR(V')$ be a supporting hyperplane of $\Omega$ at $p$, let $p\in\Omega$, let $\psi$ be the projection from $\PR(\V)\smallsetminus\{p\}$ to $\PR(V')$. The map $\psi_{|\partial\Omega}$ is a local homeomorphism onto $\PR(V')$, it is injective on $\overline F_\Omega(x)$, and $\psi(\overline F_\Omega(x))\subset \mathbb{A}':=\mathbb{A}\cap \PR(V')$. As a consequence, there exists a compact neighbourhood $W$ of $\overline F_\Omega(x)$ in $\partial\Omega$ such that $\psi_{|W}$ is an open embedding whose image lies in $\mathbb{A}'$. Moreover, there exists a compact neighbourhood $U_0$ of $x$ such that $\overline B_{\overline\Omega}(y,R)\subset W$ for any $y\in U_0$. We may assume that $U\subset U_0$.
 
 For any $y\in \psi(U)$ and $0<t\leq R$, we let $B_t(y)=\psi(\overline B_{\overline\Omega}(\psi^{-1}(y),t))$, $U'=\psi(U)$ and $x'=\psi(x)$. We want to prove that 
 \[\bigcup_{0<t<r}B_t(x')\subset \interior_{\mathbb A'} \left(\bigcup_{y\in U}B_R(y)\right).\]
 
 Fix any $t\in (0,r)$ and any affine subspace $\mathbb A_1\subset \mathbb A'$ containing $x'$ and transverse to the span of $B_R(x)$. For $s>0$ we denote $B_{\mathbb A_1}(x,s):=\{z\in \mathbb A_1 : d_{\mathbb A'}(x,z)<s\}$. For any two points $p,q\in\mathbb A'$, the difference $p-q$ is a vector of the linear space associated to the affine space $\mathbb A'$, and for any subset $E\subset \mathbb A'$ we denote $E+p-q:=\{e+p-q:e\in E\}$. To conclude the proof it is enough to find $\epsilon>0$ such that for any $z\in B_t(x)$,
 \[B_{\mathbb A_1}(x',\epsilon) + z-x'\subset \bigcup_{y\in U'\cap \mathbb A_1} B_R(y). \]
 
 By assumption that any accumulation point of $B_R(y)$ (for the Hausdorff topology) as $y$ tends to $x'$ contains $B_r(x')$, and because $t<r$, we can find $\alpha>0$ small enough so that $B_R(y)$ intersects $\mathbb A_1 + z-x$ for all $y\in \overline B_{\mathbb A_1}(x',\alpha)$ and $z\in B_t(x')$. Since $B_R$ is upper semi-continuous, the map $(y,z)\in \overline B_{\mathbb A_1}(x',\alpha)\times B_t(x')\mapsto B_R(y)\cap (\mathbb A_1+z-x)$ is also upper semi-continuous.
 
 Let us explain how the rest of the proof works in a particular case, before we proceed to the general case. Let us assume that for any $y\in B_{\mathbb A_1}(x',\alpha)$, the dimension of $B_R(y)$ is the same as that of $B_R(x')$. Fix $z\in B_t(x')$. Up to taking $\alpha$ even smaller, we may further assume that, for any $y\in B_{\mathbb A_1}(x',\alpha)$, the intersection $B_R(y)\cap (\mathbb A_1+z-x)$ is reduced to a singleton that we denote by $\{f(y)\}$. One can check that $y\mapsto B_R(y)\cap (\mathbb A_1+z-x')$ being upper semi-continuous implies that the map $f$ is continuous. Moreover, $f$ is injective since two open faces of $\Omega$ intersect if and only if they coincide. We can conclude the proof of Lemma~\ref{lemme_du_grain_de_sable} by using the Invariance of Domain theorem, which says that $f(B_{\mathbb A_1}(x',\alpha))$ is a neighbourhood of $z=f(x')$ in $\mathbb A_1+z-x'$.
 
 We go back to the general case.
 For any open subset $\cal O$ of an affine space, we denote by $\mathrm{CvxCpt}(\cal O)$ the topological space consisting of non-empty convex compact subsets of $\cal O$, endowed with the weakest topology making upper semi-continuous maps continuous. We consider the following continuous map:
 \begin{equation*}
  \begin{array}{llccl}
   f & : & \overline B_{\mathbb A_1}(x',\alpha)\times {B}_t(x') & \longrightarrow & \mathrm{CvxCpt}(\mathbb{A}_1)\\
   && (y,z) & \mapsto & \left({B}_R(y)-z+x\right)\cap \mathbb{A}_1.
  \end{array}
 \end{equation*}
 
 Note that by definition of ${B}_R$, for all $z\in {B}_t(x)$ and $y\in \overline B_{\mathbb A_1}(x',\alpha)\smallsetminus\{x\}$, we have $f(x,z)=\{x\}$ while $x\not\in f(y,z)$. Therefore we can consider $0<\epsilon<\alpha$ such that $\epsilon< d_{\mathbb{A}_1}(x,f(y,z))$ for all $z\in\overline{B}(x)$ and $y\in\partial_{\rel}\overline{B}_{\mathbb{A}_1}(x,\alpha)$.
 
 To conclude the proof of Lemma~\ref{lemme_du_grain_de_sable}, it is enough to prove that for any $z\in {B}_t(x)$,
 \[\overline{B}_{\mathbb{A}_1}(x',\epsilon)\subset \bigcup_{y\in\overline{B}_{\mathbb{A}_1}(x',\alpha)}f(y,z).\]
 It will be a consequence of the following result, whose proof we postpone until the next section.
 
 \begin{lemma}\label{Homotopy}Let $\cal O$ be an open subset of an affine space. Then the map
  \begin{equation*}
  \begin{array}{lll}
   \cal O & \longrightarrow & \mathrm{CvxCpt}(\cal O)\\
   x & \mapsto &\{x\}
  \end{array}
 \end{equation*}
 is an embedding and a weak homotopy equivalence.
 \end{lemma}

 Let us fix $z\in{B}_t(x')$ and $ p\in\overline{B}_{\mathbb{A}_1}(x',\epsilon)\smallsetminus\{x'\}$, and assume by contradiction that $p$ is not in $\bigcup_{y\in\overline{B}_{\mathbb{A}_1}(x',\alpha)}f(y,z)$. Then the continuous map
 \begin{equation*}
  \begin{array}{rcl}
   \partial_{\rel} B_{\mathbb{A}_1}(x,\epsilon') & \longrightarrow & \mathrm{CvxCpt}(\mathbb{A}_1\smallsetminus\{p\})\\
   y & \mapsto &f(y,z)
  \end{array}
 \end{equation*}
 is homotopically trivial; it is also homotopic to $y  \mapsto f(y,x')$, which is in turn homotopic to $y\mapsto\{y\}$. By Lemma~\ref{Homotopy}, this means that the inclusion $\partial_{\rel} B_{\mathbb{A}_1}(x',\alpha)\hookrightarrow \mathbb{A}_1\smallsetminus\{p\}$ is homotopically trivial. This is a contradiction because $p\in B_{\mathbb{A}_1}(x,\alpha)$.
\end{proof}
 
\subsection{Proof of Lemma~\ref{Homotopy}}\label{Ssection : equivalence d'homotopie faible}

We use the following fact, which is probably well known to experts. We recall its proof for the reader's convenience.

\begin{fait}\label{Critere pour homotopie}
 Let $p\in Y\subset X$ be a topological space, a subspace and a point. Assume that for any integer $n\geq 0$, for any continuous map $f:[0,1]^n\rightarrow X$, there exists a continuous map $H:[0,1]^{n+1}\rightarrow X$ such that :
 \begin{itemize}
  \item $H(x,0)=f(x)$ for any $x\in[0,1]^n$;
  \item $H([0,1]^n\times\{1\})\subset Y$;
  \item for any face $F\subset [0,1]^n$ (i.e. of the form $F=F_1\times\dots\times F_n$ with $F_i\in\{[0,1],\{0\},\{1\}\}$ for each $1\leq i\leq n$), if $f(F)\subset Y$ (resp. $\{p\}$) then $H(F\times[0,1])\subset Y$ (resp. $\{p\}$).
 \end{itemize}
 Then the inclusion map $\iota:Y\hookrightarrow X$ is a weak homotopy equivalence.
\end{fait}

\begin{proof}
Let $n$ be a natural number. Let us prove that $\iota_*:\pi_n(Y,p)\rightarrow\pi_n(X,p)$ is surjective. We consider a continuous map $f:[0,1]^n\rightarrow X$ which sends $\partial_{\R^n}[0,1]^n$ to $p$, we want to prove that it is homotopic, relatively to $p$, to a continuous map $[0,1]^n\rightarrow Y$ sending $\partial_{\R^n}[0,1]^n$ to $p$. The homotopy is exactly given by the map $H:[0,1]^{n+1}\rightarrow X$ provided by our assumption.

Let us prove that $\iota_*:\pi_n(Y,p)\rightarrow\pi_n(X,p)$ is injective. We consider continuous map $f:[0,1]^n\rightarrow Y$ and a homotopy $h:[0,1]^{n+1}\rightarrow X$ (sending $\partial_{\R^n}([0,1]^n)\times[0,1]$ to $p$) from $f=h_{|[0,1]^n\times\{0\}}$ to $h_{|[0,1]^n\times\{1\}}$ constant equal to $p$. By assumption we can find a continuous map $H:[0,1]^{n+2}\rightarrow X$ such that:
 \begin{itemize}
  \item $H(x,0)=h(x)$ for any $x\in[0,1]^{n+1}$.
  \item $H([0,1]^{n+1}\times\{1\})\subset Y$.
  \item For any face $F\subset [0,1]^{n+1}$ (i.e. of the form $F=F_1\times\dots\times F_{n+1}$ with $F_i\in\{[0,1],\{0\},\{1\}\}$), if $h(F)\subset Y$ (resp. $\{p\}$) then $H(F\times[0,1])\subset Y$ (resp. $\{p\}$).
 \end{itemize}
Since $h([0,1]^n\times\{0\})\subset Y$, this means that $H([0,1]^n\times\{0\}\times[0,1])\subset Y$. Then $f$ is homotopic in $Y$ to $H_{|[0,1]^n\times\{0\}\times\{1\}}$, which is homotopic in $Y$ to $H_{|[0,1]^n\times\{1\}\times\{1\}}$, which is constant equal to $p$ because $h([0,1]^n\times\{1\})=p$.
\end{proof}

\begin{proof}[Proof of Lemma~\ref{Homotopy}]
 Let an integer $n\geq 1$ and a continuous map $f:[0,1]^n\rightarrow \mathrm{CvxCpt}(\cal O)$. By continuity there is an integer $N\geq1$ such that for each $x\in \{0,\frac{1}{N},\frac{2}{N},\dots,1\}^n$ there is a convex compact set $K_x\subset \cal O$ such that for any $y\in[0,1]^n$, if $|x_i-y_i|\leq\frac{1}{N}$ for $1\leq i\leq n$, then $f(y)\subset K_x$. Fix for each $x\in \{0,\frac{1}{N},\frac{2}{N},\dots,1\}^n$ a point $p_x\in K_x$. We define for each $x\in \{0,\frac{1}{N},\frac{2}{N},\dots,\frac{N-1}{N}\}^n$, for each $y\in[0,1]^n$ and for each $t\in[0,1]$,
 \[H(x+\frac{y}{N},t)=t\sum_{\epsilon\in\{0,1\}^n}\left(\prod_{1\leq i\leq n}(1_{\epsilon_i=1}y_i+1_{\epsilon_i=0}(1-y_i))\right)p_{x+\frac{\epsilon}{N}} + (1-t)f(x+\frac{y}{N}).\]
 And finally we apply Fact~\ref{Critere pour homotopie}.
\end{proof}

\section{Proof of Theorem~\ref{thm:limset=bord}}\label{sec:demo}

Suppose by contradiction that there exists an open subset $U\subset\partial\Omega$ that does not contain any point of $\Lambda^{\prox}_{\Omega}$. 
%By a result of Vey \cite[Prop.\,3]{vey}, this implies that $U$ does not contain any  extremal point (recall that Vey's result, saying that $\Lambda_\Gamma$ coincides with the closure of the set of extremal points, is an elementary consequence of cocompactness of the action of $\Gamma$ and lower semi-continuity of $d_{\overline\Omega}$).
Take $R>0$ from Lemma~\ref{lem:ombre faible} and fix $o\in\Omega$. By Lemmas~\ref{lem:zorn} and \ref{lemme_du_grain_de_sable}, we can find $x\in U$ such that, given any compact neighbourhood $A\subset U$ of $x$, the ball $\overline B_{\overline\Omega}(x,R)$ is contained in the interior of $B_{\overline\Omega}(A,R+1)$ relative to $\partial\Omega$.

By Fact~\ref{semi-continuite superieure}, any accumulation point of $\overline B_{\overline \Omega}(y,R)$ for the Hausdorff topology, as $y$ tends to $x$, is contained in $\overline B_{\overline\Omega}(x,R)$ and hence in the interior of $B_{\overline\Omega}(A,R+1)$ relative to $\partial\Omega$.

The stereographic projection $\overline\Omega\smallsetminus\{o\} \rightarrow \partial\Omega$ sends $\overline B_\Omega(y,R)$ onto the closed shadow $\overline{\mathcal O}_R(o,y)$ for any $y\in \overline\Omega\smallsetminus \overline B_\Omega(o,R)$. By continuity of this stereographic projection, for any sequence $(y_n)_n$ in $\Omega$ converging to $x$, the sequence $(\overline B_{\Omega}(y_n,R))_n$ converges for the Hausdorff topology if and only $(\overline{\cal O}_R(o,y_n))_n$ converges, in which case they have the same limit.

Thus, for any $y\in\Omega$ close enough to $x$, the open shadow $\mathcal O_R(o,y)$ is contained in the interior of $B_{\overline\Omega}(A,R+1)$ relative to $\partial\Omega$, which contains no extremal point since $A$ contains no extremal point. Since $\mathcal O_R(o,y)\subset\partial\Omega$ is open, it does not contain any point of $\Lambda^{\prox}_\Omega$ (which is the closure of the set of extremal points by Section~\ref{sec:prox}). This contradicts Lemma~\ref{lem:ombre faible}.

\bibliographystyle{alpha-mod}
{\small \bibliography{bib}}

\Addresses

\end{document}